\documentclass[usenames,dvipsnames, 11pt]{amsart}
\usepackage[utf8]{inputenc}

\usepackage[margin=2cm,a4paper]{geometry} 
\usepackage{multirow,listings,setspace,gnuplottex,latexsym,keyval,ifthen,moreverb,lscape,forest}
\usepackage{pgf,tikz}
\usetikzlibrary{arrows}

\usepackage{amsmath, amssymb, amsthm, bm}
\usepackage{thm-restate}
\usepackage[backref]{hyperref}
\usepackage[capitalize]{cleveref}
\usepackage{color}
\usepackage{url}
\usepackage{enumitem}
\usepackage{graphicx}
\usepackage{appendix}
\usepackage{cite}
\usepackage{mathtools}

\usepackage{xcolor}
\hypersetup{
	colorlinks,
	linkcolor={red!60!black},
	citecolor={green!60!black},
	urlcolor={blue!60!black}
}
\usepackage[open,openlevel=2,atend]{bookmark}
\usepackage[abbrev,msc-links,backrefs]{amsrefs}

\def\alabel{\upshape(\makebox[\widthof{\itshape
a}][c]{\itshape \alph*}\,)}

\newtheorem{thm}{Theorem}[section]

\newtheorem{lem}[thm]{Lemma}

\newtheorem{conj}[thm]{Conjecture}

\numberwithin{thm}{section}

\theoremstyle{definition}
\newtheorem{mydef}{Definition}

\newtheorem{ques}[thm]{Question}
\newtheorem{prob}[thm]{Problem}
\newtheorem{claim}[thm]{Claim}

\newcommand{\eps}{\varepsilon}

\DeclarePairedDelimiter{\parens}{(}{)}
\DeclarePairedDelimiter{\set}{\{}{\}}

\DeclarePairedDelimiter{\size}{|}{|}

\title[Graphs with no $\ell$-path connecting equal-degree vertices]{The density of graphs with no $\ell$-path connecting equal-degree vertices: a short proof}

\author[Y. Attwa]{Yamaan Attwa}
\address{Institut f\"ur Mathematik, Freie Universit\"at Berlin, Berlin, Germany}
\email{y.attwa@fu-berlin.de}

\author[M. Azócar]{Matías Azócar Carvajal}
\address{Fachbereich Mathematik, Universit\"at Hamburg, Hamburg, Germany.}
\email{matias.azocar.carvajal@uni-hamburg.de}

\author[S. Boyadzhiyska]{Simona Boyadzhiyska}
\address{HUN-REN Alfr\'ed R\'enyi Institute of Mathematics, Budapest, Hungary.}
\email{simona@renyi.hu}

\author[T. Pierron]{Théo Pierron}
\address{Univ Lyon, UCBL, CNRS, INSA Lyon, LIRIS, UMR5205, F-69622 Villeurbanne, France.}
\email{theo.pierron@univ-lyon1.fr}

\author[A. Taraz]{Anusch Taraz}
\address{Hamburg University of Technology, Institute of Mathematics, Hamburg, Germany.}
\email{taraz@tuhh.de}

\begin{document}

\begin{abstract}
    Addressing a question posed by Chen and Ma from an asymptotic point of view, we present a short proof for the edge density needed to guarantee that two vertices of the same degree are connected by a path of a fixed length. In particular, we show that for any sufficiently large graph, a density of at least $1/2+o(1)$ enforces the existence of two such vertices. This bound is tight for paths of odd length.
\end{abstract}

\maketitle

\setlength{\marginparwidth}{40pt}
\section{Introduction}

\bigskip

It is a simple exercise to show that every graph contains a pair of vertices of the same degree. Being so, it is only natural to ask: under what conditions can we expect such vertices to satisfy additional properties?
Erd\H{o}s and Hajnal~\cite{erdos1991problems}
formulated the following question of this flavor:
\begin{prob}
    Is it true that every $(2n + 1)$-vertex graph with $n^2 + n + 1$ edges contains two vertices of the same degree which are joined by a path of length three?
\end{prob}
As usual, the length of a path is the number of edges it contains.

\medskip

The bound of $n^2+n+1$ is tight, as demonstrated by the complete bipartite graph $K_{n,n+1}$.
In a recent paper, Chen and Ma~\cite{chen2026problem} resolved this problem, in a very strong sense, for large $n$. In fact, they also handled the case where the graph has an even number of vertices, characterizing the unique extremal construction in both situations.
\begin{thm}\label{thm:chen-ma}The following holds.\hfill
    \begin{enumerate}[label={\alabel}]
        \item Let $n \geq 600$. The unique $(2n + 1)$-vertex graph with at least $n^2 + n$ edges that does not contain two vertices of the same degree joined by a path of length three is the complete bipartite graph~$K_{n,n+1}$.\label{thm:cm-odd}
        \item There exists an integer $n_0 > 0$ such that the following holds for all $n \geq n_0$. The unique $2n$-vertex graph with at least $n^2-1$ edges that does not contain two vertices of the same degree joined by a path of length three is the complete bipartite graph $K_{n-1,n+1}$.\label{thm:cm-even}
    \end{enumerate}
\end{thm}
This result was later improved by Liu and Zeng~\cite{liu2025complement}, who showed that part~\ref{thm:cm-odd} holds for all $n\geq 2$ and part~\ref{thm:cm-even} is true for all $n\geq 3$.
\medskip

In their paper, Chen and Ma~\cite{chen2026problem} defined the following extremal function, generalizing the Erd\H{o}s--Hajnal problem.

\begin{mydef}
    For any positive integers $\ell$ and $n$, let $p_\ell(n)$ denote the maximum number of edges in an $n$-vertex graph $G$ that contains no two vertices of equal degree connected by a path of length $\ell$.
\end{mydef}

The results in~\cites{chen2026problem,liu2025complement} thus show that $p_3(2n+1) = n^2+n$ for all $n\geq 2$ and $p_3(2n) = n^2-1$ for all~$n\geq 3$. Chen and Ma further showed that
\begin{align*}
    p_1(n) = \frac{n^2}{2} - \frac{n\sqrt{2n}}{3} + O(n) \quad \text{and} \quad p_2(2n) = \frac{n(n+1)}{2}.
\end{align*}
Additionally, they conjectured that the behavior of $p_\ell(2n+1)$ should be the same for any odd $\ell\geq 3$.

\begin{conj}\label{conj:chen-ma}
    For any odd integer $\ell\geq 3$ and sufficiently large $n$, it holds that \hfill ${p_\ell(2n + 1) = n^2 + n}$.
\end{conj}

In a very recent paper, Liu and Zeng~\cite{liu2026paths} tackled the case $\ell=5$, once again in a strong sense, proving an analog of both parts of \cref{thm:chen-ma} (for $n\geq 11$ in part~\ref{thm:cm-odd} and $n\geq 13$ in~\ref{thm:cm-even}). The conjecture was subsequently resolved in full by Zhao, Wang, and Lu~\cite{zhao2026generalization}, who once again proved that $K_{n,n+1}$ is the unique extremal example when the number of vertices is odd. The case of an even number of vertices was not addressed directly in~\cite{zhao2026generalization}, even though the authors comment that their methods likely extend to this setting.

For even values of $\ell$, our knowledge is much more limited. The \emph{half graph} on $2n$ vertices is defined as the bipartite graph with vertex classes $\set{u_1,\dots, u_n}$ and $\set{v_1,\dots, v_n}$, where $u_iv_j$ is an edge if and only if $i\leq j$. It is not difficult to check that in the half graph no vertices of the same degree are connected by an even-length path, implying that
$p_\ell(2n) \geq  \frac{n(n+1)}{2}$ for all even $\ell\geq 2$. As far as we are aware, nothing more is known about this function when $\ell$ is even, though in~\cite{liu2026paths} Liu and Zeng refer to upcoming work showing that $p_4(2n) \leq (1 - \eps + o(1))n^2$ for some constant $\eps \in (0, 1/2]$ and
sufficiently large $n$.

\medskip
In all previous work on the problem, the arguments have been rather technical and involved.
The goal of this note is to provide a short proof of an asymptotic version of \cref{conj:chen-ma}. Our main result addresses the remaining cases by showing that the number of edges in a graph not containing a pair of vertices connected by a path of some fixed length $\ell\geq 6$, odd or even, is always at most $\frac12\binom{n}{2}$ asymptotically. Our analysis is somewhat reminiscent of that provided in~\cite{zhao2026generalization} but it is much simpler.
\begin{thm}
    \label{thm:asymptotic}
    For every $\ell\geq 6$ and every sufficiently large $n$, we have $p_\ell(n) \leq \parens*{\frac{1}{4} + o(1)}n^2$.
\end{thm}

Clearly, this bound is tight when $\ell$ is odd. We tend to believe that it should not be tight for even $\ell$ and in that case the maximum number of edges should be closer to $\frac14\binom{n}{2}$. We pose this as a question.

\begin{ques}
    Is it true that $p_{\ell}(n) = (\frac{1}{8}+o(1))n^2$ for all even $\ell$?
\end{ques}

\subsection*{Notation}
Our notation is mostly standard. An $xy$-path is a path connecting vertices $x$ and $y$. Let~$G$ be a graph and $A,B\subseteq V(G)$ be any subsets (not necessarily disjoint). We write $G[A]$ for the subgraph of $G$ induced by $A$, that is, the subgraph obtained by deleting the vertices not in $A$. The neighborhood of $A$, denoted~$N(A)$, refers to the set~$\set{v\in V(G)\,:\, va\in E(G) \text{ for some } a\in A}$. We write $E(A,B)$ for the set $\set{ab\in E(G) \,:\, a\in A, b\in B}$.

\section{The proof}

This section is devoted to the proof of \cref{thm:asymptotic}. Throughout this section, we will always assume that $\ell\geq 2$ is a fixed integer, $n$ is sufficiently large with respect to $\ell$, and $G$ is an $n$-vertex graph. The proof of \cref{thm:asymptotic} relies on the following technical lemma, which allows us to find long paths when~$G$ has too many vertices of large enough degree. We consider several cases depending on whether we look for paths of even or odd length.

\begin{lem}\label{lem:large-vertices-paths}
    Let $D \geq \frac{n}{2} + k+1$ be an integer. Suppose $G$ contains a set $B = \set{b_1,\dots, b_t}$ of $2\leq t\leq k$ vertices of degree at least $D$, and let $x,y\notin B$ be distinct vertices of degree at least $n-D+t+4$. Then:
    \begin{enumerate}[label=\alabel]
        \item There exists an $xy$-path of length $2t+2$.\label{lvp:simple-path}
        \item If $G[B]$ contains an edge, then there exists an $xy$-path of length $2t+1$. \label{lvp:b-independent}
        \item If  $E(N(b_i),N(b_j))\neq \emptyset$ for some $i\neq j$, then there exists an $xy$-path of length $2t+3$. \label{lvp:independent-neighborhoods}
    \end{enumerate}
\end{lem}
\begin{proof}\hfill
    \begin{enumerate}[leftmargin=*,label=\alabel]
        \item  Set $b_0=x$ and $b_{t+1}=y$. We iteratively find a family of pairwise distinct vertices $a_1,\ldots,a_{t+1}\notin B\cup\{x,y\}$ such that $a_i$ is a common neighbor of $b_{i-1}$ and~$b_i$ for each $1\leq i\leq t$. Once this is done, $xa_1b_1a_2\ldots a_t b_t a_{t+1}y$ will form the desired~$xy$-path of length $2t+2$.

              We first choose $a_1$ and $a_{t+1}$. Observe that $\deg(b_i)+\deg(b_{i+1})\geq n+t+4$ for $i\in\{0,t\}$. Hence $|N(b_i)\cap N(b_{i+1})|\geq t+4$, and thus this set contains two distinct vertices not in $B\cup\{x,y\}$.

              Now, for $1<i\leq t-1$, if $a_1,\ldots,a_i$ are already selected, we similarly find that $|N(b_i)\cap N(b_{i+1})|\geq 2D-n\geq 2t+2$. So we can choose the next vertex ${a_{i+1}\in N(b_i)\cap N(b_{i+1})}$ outside of~$B\cup\{x,y,a_1,\ldots,a_i, a_{t+1}\}$; indeed, since~${b_i\notin N(b_i)}$ and~$b_{i+1}\notin N(b_{i+1})$, the latter set contains at most $2t$ vertices of~$N(b_i)\cap N(b_{i+1})$.
        \item Assume without loss of generality that $b_1b_2$ is an edge. We choose $a_1,\ldots,a_{t+1}$ as in the previous case, but build a path using the edge $b_1b_2$ instead of the vertex~$a_2$.
        \item Assume without loss of generality that there is an edge $a_2a_2'$ with $a_2\in N(b_1)$ and~$a_2'\in N(b_2)$. In this case, we choose $a_1 \in (N(b_1)\cap N(x))\setminus (B\cup\{x,y,a_2,a'_2\})$, since this set contains at least $t+4-(t+2)\geq 2$ vertices. Similarly, we choose~$a_{t+1} \in (N(b_t)\cap N(y))\setminus (B\cup\{x,y,a_2,a'_2\})$. Since, in the previous items, we had at least two choices for the vertices $a_3,\dots, a_t$, we can also choose them distinct from $a'_2$. We thus obtain the path $xa_1b_1a_2a_2'b_2a_3\dots a_t b_t a_{t+1}y$ of length~$2t+3$. \qedhere
    \end{enumerate}
\end{proof}

Using this lemma, we can complete the proof of \cref{thm:asymptotic}. The proof is split into two cases, depending on the parity of $t$. We start with the less involved even case.

\subsection{Even-length paths}

Assume that $G$ does not contain two vertices of the same degree connected by a path of length $\ell=2k$. Our goal is to show that $e(G) = n^2/4+O(n)$. Label the vertices of $G$ as~$v_1,\ldots,v_n$ in non-increasing degree order, and let $\Delta_k = \deg(v_{k-1})$.  We consider two cases depending on whether $\Delta_k$ is large or not.
\smallskip

\noindent\textbf{Case 1: $\Delta_k < \frac{n}{2} + k+1$}.
Summing up the degrees of the vertices, and upper-bounding~$\deg(v_i)$ by~$n$ if $i<k-1$ and by $\Delta_k$ otherwise, we obtain
\begin{align*}
    2e(G) \leq (k-2)n+(n-k+2)\parens*{\frac{n}{2} + k} \leq \frac{n^2}{2} + \frac32 kn,
\end{align*}
as claimed.

\smallskip
\noindent\textbf{Case 2: $\Delta_k \geq \frac{n}{2} + k+1$}.
Our goal is to apply Lemma~\ref{lem:large-vertices-paths}~\ref{lvp:simple-path}, with $t=k-1$, $B=\{v_1,\ldots,v_{k-1}\}$, and~$D=\Delta_k$. If there is at most one vertex $x\notin B$ of degree at least $n-\Delta_k+(k-1)+4 = n-\Delta_k+k+3$, we are done by a similar argument as above. So we may assume there exist at least two of them, say~$x$ and $y$. By \cref{lem:large-vertices-paths}, there exists an $xy$-path of length $2(k-1)+2 = 2k$, so we conclude that~$x$ and~$y$ must have distinct degrees.

Again, summing up the degrees of the vertices, we obtain that $2e(G)$ is bounded above by
\begin{align}
     & \underbrace{(k-1)n}_{v_i\text{ for } i< k} + \underbrace{\sum_{i=n-\Delta_k+k+3}^{\Delta_k}i}_{\text{at most one vertex of each degree }}+\underbrace{\parens*{n - \parens*{\Delta_k-(n-\Delta_k+k+3)}+1}(n-\Delta_k+k+2)}_{\text{remaining vertices}}\label{eq:edge-sum} \\
     & = (k-1)n + \binom{\Delta_k+1}{2}-\binom{n-\Delta_k+k+3}{2}+(2n-2\Delta_k+k+4)\parens*{n-\Delta_k+k+2}\notag                                                                                                                                                               \\
     & \leq (k-1)n + \binom{\Delta_k+1}{2}-\binom{n-\Delta_k+k+3}{2}+(n-k+2)\parens*{n-\Delta_k+k+2}\notag                                                                                                                                                                       \\
     & = \frac{\Delta_k^2}{2}-\frac{(n-\Delta_k)^2}{2}+n(n-\Delta_k)+O(n)\notag = \frac{n^2}{2} + O(n), \notag
\end{align}
which concludes the even case.

\subsection{Odd-length paths}

Assume now that $G$ does not contain two vertices of the same degree connected by a path of length $\ell = 2k+1$ with $k\geq 3$. We again label the vertices $v_1,\ldots,v_n$ by non-increasing degree and set $\Delta_k = \deg(v_k)$.

Let $D= \frac{n}{2} + k+1$ and $B = \set{v\in V(G)\,:\, \deg(v)\geq D}$. Note that $2e(G) \leq n\size{B} + D(n-\size{B})$. If $\size{B} = o(n)$, then $e(G) \leq \frac{n^2}{4}+o(n^2)$ and we are done. So we may assume that $\size{B} = \omega(n)$ and in particular that $\Delta_k\geq D$. Additionally, write $R = V(G)\setminus (B\cup N(B))$.

\smallskip
\noindent\textbf{Case 1: $G[B]$ contains an edge $uw$ or $E(N(u),N(w))\neq\emptyset$ for some distinct $u,w\in B$}.

Suppose first that $G[B]$ contains the edge $uw$. Let $x,y$ be two vertices of degree at least $n-\Delta_k+k+4$. Select $k-2$ vertices $b_1,\dots, b_{k-2}\in B\setminus\set{u,w,x,y}$, and write $b_{k-1} = u$, $b_{k}=w$. By \cref{lem:large-vertices-paths}~\ref{lvp:b-independent}, $G$ must contain an $xy$-path of length $2k+1$, and hence $x$ and $y$ have distinct degrees. In particular, no two vertices of $B$ (except possibly $u,w$) share the same degree.
Then proceeding similarly to~\eqref{eq:edge-sum}, where we allow for the degrees of $u$ and $w$ to be as large as $n$, we obtain
\begin{align*}
    2e(G) & \leq (k+2)n + \sum_{i=n-\Delta_k+k+4}^{\Delta_k}i+\parens*{n - \parens*{\Delta_k-(n-\Delta_k+k+4)}+1}\cdot (n-\Delta_k+k+3) \\
          & \leq \frac{n^2}{2} + O(n),
\end{align*}
using that $n\geq \Delta_k\geq D\geq n/2$.

The latter case, where $E(N(u),N(w))\neq\emptyset$ for some distinct $u,w\in B$, is similar, except we now take   $b_1,\dots, b_{k-3}\in B\setminus\set{u,w,x,y}$, and write $b_{k-2} = u$, $b_{k-1}=w$ and apply \cref{lem:large-vertices-paths}~\ref{lvp:independent-neighborhoods}.
\smallskip

\noindent\textbf{Case 2: $B$ is an independent set and $E(N(u),N(w))=\emptyset$ for any distinct $u,w\in B$.}

In this case, $B\cap N(B) = \emptyset$. Partition $N(B)$ into sets $X\cup \parens*{\bigcup_{v\in B}Y_v}$, where $X = \set{w\in V(G)\,:\, \size{N(w)\cap B}\geq 2}$ and $Y_v = \set{w\in V(G) \,:\, N(w)\cap B = \set{v}}$ for all $v\in B$. Note that the only edges within $N(B)$ are fully contained in the sets $Y_v$. We now claim that $X$ is relatively large.

\begin{claim}
    \label{cl:sizeofX}
    $|X|\geq \frac{|B|-2}{|B|-1}\cdot\frac{n}{2}$.
\end{claim}

\begin{proof}
    We proceed by double counting the edges between $B$ and $X$. On the one hand, each vertex in~$X$ is adjacent to at most $\size{B}$ vertices in $B$, so $|E(B,X)|\leq \size{X}\cdot \size{B}$. On the other hand, we have
    \begin{align*}
        \size{E(B,X)} & = \size{E(B,N(B))}-\sum_{v\in B} \size{E(B,Y_v)}\geq D\size{B} - \sum_{v\in B}\size{Y_v} \\
                      & \geq \frac n2\size{B} - (n-|B|-|X|-|R|)\geq (\size{B}-2)\frac n2+\size{X},
    \end{align*}
    where in the first inequality we used that every vertex in  $Y_v$ has at most one edge to $B$.
    Combining these two bounds proves the claim.
\end{proof}

We now conclude by estimating the number of edges of $G$. For this, recall that $B$ is an independent set, the only edges within $N(B)$ can occur inside the individual sets $Y_v$, and every vertex in $R$ has degree at most $D$ by definition. Hence
\begin{align}
    e(G) & \leq \size{B}\cdot \size{X}+\sum_{v\in B} \parens*{\size{Y_v}+{\size{Y_v}\choose 2}}+ D\size{R}\notag                                            \\
         & \leq \size{B}\cdot \size{X} + \binom{\sum_{v}\size{Y_v}}{2} + \frac n2\parens*{n-\size{B}-\size{N(B)}} +O(n)\notag                               \\
         & \leq \size{B}\cdot \size{X} + \binom{\size{N(B)}-\size{X}}{2} + \frac n2\parens*{n-\size{B}-\size{N(B)}} +O(n),\label{eq:edge_bound_quadratic_X}
\end{align}
where for the second inequality we used that $\sum_{v\in B} {\size{Y_v}\choose 2}\leq \binom{\sum_{v}\size{Y_v}}{2}$, and that $R = V(G)\setminus (B\cup N(B))$ and $B\cap N(B)=\emptyset$.
Note that, with respect to $\size{X}$, the expression in~\eqref{eq:edge_bound_quadratic_X} is an upward-opening parabola, so its value is maximized at some extreme point for $\size{X}$. By Claim~\ref{cl:sizeofX} and the fact that~$X\subseteq N(B)$, we know that $\frac{(|B|-2)}{(|B|-1)} \cdot \frac n2 \leq \size{X}\leq \size{N(B)}$. We consider these two cases separately. Observe that $\frac{n}{2} \leq D\leq \size{N(B)}\leq n-\size{B}$, so in particular $\size{B} \leq n-D < \frac n2$.
\medskip

\noindent\textbf{Case 1: $\size{X} = \size{N(B)}$}.  Then~\eqref{eq:edge_bound_quadratic_X} becomes
\begin{align*}
     & \size{B}\cdot \size{N(B)} + \frac n2\parens*{n-\size{B}-\size{N(B)}} +O(n)                                      \\&= \size{N(B)} \parens*{\size{B} - \frac n2} +\frac{n^2}{2} - \frac{n}{2}\size{B} + O(n)\\
     & \leq  \frac n2\parens*{\size{B} - \frac n2} +\frac{n^2}{2} - \frac{n}{2}\size{B} + O(n) = \frac {n^2}{4} +O(n),
\end{align*}
as required.

\smallskip
\noindent\textbf{Case 2: $\size{X} = \frac{(|B|-2)}{(|B|-1)} \cdot \frac n2$}. Note that in this case~\eqref{eq:edge_bound_quadratic_X} is also an upward-opening parabola with respect to $\size{N(B)}$. Hence, it is maximized when either $\size{N(B)} = D$ or $\size{N(B)} = n-\size{B} \leq n$. In the latter case,~\eqref{eq:edge_bound_quadratic_X} is upper bounded by
\begin{align*}
     & \size{B}\cdot \size{X} + \binom{n-\size{B}-\size{X}}{2} +O(n) \leq \size{B}\cdot \size{X} + \frac{\parens*{n-\size{B}-\size{X}}^2}{2} +O(n) \\
     & \leq \frac{\size{B}^2}{(\size{B}-1)^2}\cdot\frac{n^2}{8} + \frac{\size{B}^2}{2} +O(n) \leq \parens*{1+o(1)}\frac{n^2}{4},
\end{align*}
since $\size{B} = \omega(n)$ and $\size{B}\leq \frac n2$.

\smallskip
In the  case $\size{N(B)} = D$, we bound~\eqref{eq:edge_bound_quadratic_X} by
\begin{align*}
     & \size{B}\cdot \size{X} + \binom{n/2+k+1-\size{X}}{2} + \frac n2\parens*{n-\size{B}-\size{N(B)}} +O(n)                                                         \\
     & \leq \size{B}\cdot \frac{n}{2} + \frac{\parens*{\frac{n}{2(\size{B}-1)}}^2}{2} + \frac n2\parens*{\frac{n}{2}-\size{B}} +O(n) =\parens*{1+o(1)}\frac{n^2}{4},
\end{align*}
where we used that $\size{X}\leq \frac n2$ and $\size{B}=\omega(n)$. This concludes the proof for the odd-length case.

\section*{Acknowledgments}
\noindent (YA) Research supported by the DFG under Germany´s Excellence Strategy – The Berlin
Mathematics Research Center MATH+ (EXC-2046/1, project ID: 390685689).

\noindent(MA) Research supported by ANID and DAAD under ANID-PFCHA/Doctorado Acuerdo Bilateral DAAD Becas Chile/2023-62230021.

\noindent(SB) Research supported by ERC Advanced Grants ``GeoScape'', no.~882971 and ``ERMiD'', no.~101054936.

\noindent(TP) Research supported by ANR Grant ENEDISC ANR-24-CE48-7768-01.

\bibliographystyle{amsplain}
\bibliography{biblio.bib}
\end{document}